\documentclass[11pt, oneside, a4paper, reqno]{amsart}
\usepackage[foot]{amsaddr}
\usepackage{amsmath, amsfonts, amssymb, dsfont, nicefrac,bm,upgreek,mathtools,verbatim, color, amsopn,esint}
\usepackage[final]{hyperref}
\usepackage[mathscr]{eucal}
\usepackage[normalem]{ulem}
\usepackage[T1]{fontenc}
\usepackage{bigints}
\usepackage[margin=1.1in, foot=0.25in]{geometry}
\hypersetup{linkcolor=blue, colorlinks=true
,citecolor = red}
\allowdisplaybreaks

\newtheorem{theorem}{Theorem}[section]
\newtheorem{lemma}[theorem]{Lemma}

\theoremstyle{definition}
\newtheorem{definition}[theorem]{Definition}
\newtheorem{remark}[theorem]{Remark}

\numberwithin{theorem}{section}
\numberwithin{equation}{section}

\newcommand{\D}{\mathcal{D}}
\newcommand{\R}{\mathbb{R}}

\newcommand{\overbar}[1]{\mkern 1.5mu\overline{\mkern-1.5mu#1\mkern-1.5mu}\mkern 1.5mu}

\usepackage[capitalize,nameinlink]{cleveref}

\makeatletter
\def\@settitle{
  \begin{center}
    \normalfont\LARGE\bfseries
    \@title
  \end{center}
}
\def\@setauthors{
  \begingroup
  \trivlist
  \centering\large \@topsep30\p@\relax
  \advance\@topsep by -\baselineskip
  \item\relax
  \author@andify\authors
  \def\\{\protect\linebreak}
  {\authors}
  \endtrivlist
  \endgroup
}
\makeatother
\newcommand{\ttl}{\Large Weighted Sobolev inequalities and superlinear elliptic problems on exterior domains}
\begin{document}
\title[Weighted Sobolev inequalities and superlinear problems]
{\ttl}

\author{Ardra A$^1$ } 
\author{Ameerraja Ansari$^2$}
\author{Anumol Joseph$^3$}
\author{Lakshmi Sankar$^{4, *}$}
\address{$^*$ Corresponding author}
\address{$^{1,2,4}$ Department of Mathematics, Indian Institute of Technology Palakkad, Kerala-678623, India}
\address{$^3$ Department of Mathematics, SRM University-AP, Amaravati-  522240, Andhra Pradesh, India}
\email{ardra.math@gmail.com}
\email{212214001@smail.iitpkd.ac.in}
\email{anumol.j@srmap.edu.in}
\email{lakshmi@iitpkd.ac.in}

\begin{abstract} 
Let $B_1 ^c = \{ x\in \mathbb{R}^N: |x|>1 \}, N \geq 2$, and $\mathcal{D}^{1,N}_0(B^c_1)$, be the Beppo-Levi space. We prove that  $\mathcal{D}^{1,N}_0(B^c_1)$ is compactly embedded into the weighted Lebesgue space $L^r(B_1^c;K(x))$ for all $r\in[1,\infty)$ for an appropriate class of weight functions $K$. 
As an application, we prove the existence of a positive solution to a superlinear semipositone problem on $B_1 ^c$ in $\mathbb{R}^2$. We also establish boundedness and  regularity of solutions of certain boundary value problems and derive their Green's function representation. 
\end{abstract}

\subjclass[2020]{Primary ; Secondary }

\maketitle

\section{Introduction}

Let $B_1 ^c = \{ x\in \mathbb{R}^N: |x|>1 \}$, $N \geq 2$, and $\mathcal{D}^{1,p}_0(B^c_1), p>1$ be the Beppo-Levi space, which is the completion of compactly supported smooth functions on $B_1 ^c$ with respect to the norm $||u||= (\int_{B_1^c} |\nabla u|^p dx)^{\frac{1}{p}}$. For $1<p<N$, the classical Sobolev embedding theorem guarantees that $\mathcal{D}^{1,p}_0(B^c_1)$ is continuously embedded in $L^{p^*}(B_1^c),$ where $p^*=\frac{Np}{N-p}$ is the critical Sobolev exponent.
It is also known that $\mathcal{D}^{1,p}_0(B^c_1)$ is a well defined reflexive Banach space for all $p >1$ (see \cite{MR3347486}, \cite{MR1454361}). Recently, there have been several studies on the continuous and compact embeddings of $\mathcal{D}^{1,p}_0(B^c_1)$.
 In this work, we are particularly interested in the embeddings of $\mathcal{D}^{1,p}_0(B^c_1)$, in the borderline case $p=N$. With this in mind, we now review some key results from the literature that address $p=N$.

It was shown in \cite{MR1454361} (see Theorems I.2.7, I.2.16) that $\mathcal{D}^{1,N}_0(B^c_1)$ coincides with
$\mathcal{\hat{D}}^{1,N}_0(B^c_1)$, given by
\begin{eqnarray*}
\mathcal{\hat{D}}^{1,N}_0(B^c_1) &=& \{u \in L^{1, N}(B_1^c) : u \in L^N (B_1^c \cap B_R), \forall R >1, \mbox {and} \\
& & \eta u \in W^{1, N}_0(B_1^c) \mbox{~for any~} \eta \in C_c^\infty(\mathbb{R}^N) \},
\end{eqnarray*}
where $B_R$ is a ball of radius $R$, and 
$$L^{1, N}(B_1^c) =\{u \in L^1_{loc}(B_1^c) : \nabla u \in [L^N(\Omega)]^N\}.$$
Here, $W^{1, N}_0(B_1^c)$ denotes the standard Sobolev space. 
Using this characterisation, it can be seen that a function in $\mathcal{D}^{1,N}_0(B^c_1)$ may not be in $L^q(B_1^c)$ for any $q \geq 1$. For example, when $N=2$, the function $u(x)=1-\frac{1}{|x|^2}, x \in B_1^c$ belongs to $\mathcal{\hat{D}}^{1,2}_0(B^c_1)$, but $u \notin L^q(B_1^c)$ for any $q \geq 1$. 

In \cite{MR3347486}, the authors proved the compact embedding of  $\mathcal{D}^{1,N}_0(B^c_1)$ into $L^N(B_1^c, K(|x|))$ for $K$ belonging to the weighted Lebesgue space $L^1((1, \infty); (r \log r)^{N-1})$. The case $N=2$ was considered in \cite{MR3925556}, wherein the Kelvin transform $\mathcal{K}$, defined by $\hat{u}(x)=(\mathcal{K} u)(x)=u(\frac{x}{|x|^2})$, was shown to be an order-preserving, isometric isomorphism between $H=H_0^1(B_1)$ and  $\mathcal{D}^{1,2}_0(B^c_1)$ (see Theorem 2.6). Using this isomorphism, the authors further established a continuous embedding of $\mathcal{D}^{1,2}_0(B^c_1)$ into the weighted Lebesgue space $L^r(B_1^c, K(x))$, for every $1 \leq r < \infty$ with weight function $K(x)=\frac{1}{|x|^4}$ (see Proposition 2.9). In the first part of this paper, we prove a compact embedding of $\mathcal{D}^{1,N}_0(B^c_1)$ into $L^r(B_1^c, K(x))$, for any $r$ with $1 \leq r < \infty$, for a broader class of weight functions $K(x)$ that includes the weight $\frac{1}{|x|^4}$.  Our approach relies on techniques from the celebrated work of Caffarelli–Kohn–Nirenberg \cite{MR768824}. 
The precise decay condition assumed on the weight function $K(x)$ is as follows:
\begin{enumerate}
\item[(H1)] There exist $C>0$ and $\gamma >N$ such that $K(x) \leq \frac{C}{|x|^{\gamma}}$ for $|x| >1$.
\end{enumerate}
To establish the embedding results, we first derive an interpolation inequality for functions in $C_c^\infty(B_1^c)$, the space of all compactly supported smooth functions on $B_1^c$, which we state below. By a standard density argument, this inequality implies the continuous embedding of $\mathcal{D}^{1,N}_0(B^c_1)$ into $L^r(B_1^c, K(x))$, for any $r$ with $N < r < \infty$.
\begin{theorem}
\label{inequality}
Let $B_1^c$ be the exterior of the unit ball in $\mathbb{R}^N, N \geq 2$. Assume that $0 < a < 1,$ and $ \theta >0$, and let $\delta=-(N+\theta(1-a))$. 
Then
\begin{equation}
\left( \int_{B_1^c}|x|^{-(N+\theta)}|u|^r \right)^\frac{1}{r}\leq C \left( \int_{B_1^c}|\nabla u|^N \right)^\frac{a}{N} \left(\int_{B_1^c}|x|^{-(N+\theta)}|u|^N \right)^\frac{(1-a)}{N}+C \left( \int_{B_1^c}|x|^\delta|u|^N \right)^\frac{1}{N}
\end{equation}
for all $u\in C_c^\infty(B_1^c)$ and $N<r<\infty$.
\end{theorem}
Using Theorem \ref{inequality}, and the known embeddings, we prove the following.
\begin{theorem}
\label{embedding} \label{Theorem:Compact embedding} 
Let $B_1^c$ be the exterior of the unit ball in $\mathbb{R}^N, N \geq 2$, and $K : B_1^c \to \mathbb{R}^+$ be a continuous function satisfying $(H1)$. Then the Beppo-Levi space $\mathcal{D}^{1,N}_0(B^c_1)$ is continuously embedded into $L^r(B_1^c;K(x))$ and for all $r\in [1, \infty)$. In addition, this embedding is compact. 
\end{theorem}

Theorem \ref{embedding} can be used to study various elliptic PDE's on the exterior of a ball. In particular, we consider the superlinear semipositone problem
\begin{equation} \label{Problem 1}
\begin{cases}
- \Delta u &= \lambda K(x) f(u) \hspace{.1cm} \mbox { in } B_1 ^c, \\
u(x)&=0 \hspace{1.6cm} \mbox { on }  \partial B_1, \\
\end{cases}
\end{equation}
where $B_1 ^c = \{ x\in \mathbb{R}^2 : |x|>1 \}$, $\lambda$ is a positive parameter, and $K: B_1 ^c \rightarrow \mathbb{R}^{+}$, $f:[0,\infty) \rightarrow \mathbb{R}$ belong to a class of continuous functions with $f(0) <0$ and $\lim_{s \rightarrow \infty} \frac{f(s)}{s}=\infty$. 

Problems of the form \eqref{Problem 1} with $f(0)<0$ are known as semipositone problems in the literature. Establishing the existence of positive solutions to semipositone problems is known to be both challenging and interesting, since the maximum principle cannot be applied directly to guarantee the positivity of solutions. We refer the reader to \cite{MR2328697}, \cite{MR2636416}, where the existence of a positive solution for small values of the parameter $\lambda$ was established for superlinear semipositone problems on bounded domains. In addition, a superlinear semipositone problem on $\mathbb{R}^N$ was studied in \cite{MR1219715}. 

Recently, there have been several studies on the existence of positive solutions to superlinear semipositone problems on exterior domains in $\mathbb{R}^N, N > 2$ (see \cite{MR3801828}, \cite{MR3415737}, \cite{MR4500097}, \cite{MR3782023}). In \cite{MR3415737}, \cite{MR3782023}, the authors established the existence of a radial solution to \eqref{Problem 1} under the assumption that the weight function $K$ is radial. A superlinear semipositone system was studied in \cite{MR3801828}. Very recently, in  \cite{MR4500097}, the authors extended the studies in \cite{MR3415737}, \cite{MR3782023} by allowing the weight function to be non radial and proved the existence of a positive solution to \eqref{Problem 1}. We note that these existence results, including those for radial solutions, were proved on exterior domains in $\mathbb{R}^N, N > 2$. In this paper, we prove the existence of a solution to \eqref{Problem 1} on the exterior of a ball in $\mathbb{R}^2$, using the embeddings established in Theorem \ref{Theorem:Compact embedding} and the Mountain Pass Lemma. The positivity of the solution is then proved by analyzing the properties of the solution. 

It is well known that solutions to problems of the form \eqref{Problem 1}  on exterior domains in $\mathbb{R}^2$ have different asymptotic behavior compared to solutions of similar problems on exterior domains in $\mathbb{R}^N, N>2$. Specifically, when $f\geq 0$, it was proved in \cite{MR1849197} that positive solutions to \eqref{Problem 1} on the exterior of a ball in $\mathbb{R}^2$ are bounded away from zero, whereas when $N >2,$ \eqref{Problem 1} has solutions that decay to zero at infinity. Even in the case where $f(0)<0,$ the solutions obtained for \eqref{Problem 1} in \cite{MR3801828}, \cite{MR3415737}, \cite{MR4500097}, \cite{MR3782023} decay to zero at infinity. We will discuss the boundedness and regularity of solutions to \eqref{Problem 1} on the exterior of a ball in $\mathbb{R}^2$. In order to establish the positivity of the solution, we also prove that the solutions to  \eqref{Problem 1} has a Green's function representation.

 We now state our assumptions on $K, f$. Both $K$ and $f$ assumed to be H\"older continuous with $K$ satisfying $(H_1)$. Define $f(t) = f(0)$ for $t < 0$, and rewrite $f$ as, $f(t) = g(t)+f(0)$ so that $g(t) = 0 $ for all $t \leq 0$. Additionally, assume:
\begin{enumerate}
\item[(H2)] $g(t) \geq 0$ for all $t \geq 0$.
\item[(H3)]  There exist $A, B >0 $, and $s >1$ such that $At ^s \leq g(t) \leq B t^s $ for all $t \geq 0$.
\item[(H4)] For $\mu = s+1$, $\mu G(t) \leq t g(t)$ for all $t \geq 0$, where $G(t) = \int_0 ^t g(\xi) d \xi .$ 
\end{enumerate}
We give the definition of a weak solution to \eqref{Problem 1} below.
\begin{definition}
We say $u \in \D_0^{1,2}(B_1^c)$ is a (weak) solution of the problem (\ref{Problem 1}) if $u$ satisfies 
\begin{equation} \label{weak solution}
\int_{B_1 ^c} \nabla u \cdot \nabla v dx = \lambda \int_{B_1^c} K(x) f(u) v dx, \mbox{ for all } v \in \D^{1,2}_0(B_1^c).
\end{equation}
\end{definition}
We have the following existence result.
\begin{theorem}
\label{existence}
Let $K: B_1 ^c \rightarrow \mathbb{R}^{+}$, $f:[0,\infty) \rightarrow \mathbb{R}$ be H\"older continuous functions. Assume that $K$ satisfies $(H1)$ and $f$ satisfies $(H2)$-$(H4)$. Then, there exist a $\overbar{\lambda} > 0$ and $0<\alpha<1$ such that for $\lambda \leq \overbar{\lambda}$, (\ref{Problem 1}) has a  solution $u \in C^2(B_1^c) \cap C^{1, \alpha}_{loc}(\overbar{B_1^c})$.
\end{theorem}
Next, we state a general regularity result for the solutions to \eqref{Problem 1}, assuming only that $K, f$ are H\"older continuous, $K$ satisfies (H1) and $f$ satisfies the following assumption.
\begin{itemize}
\item [$(\tilde{H3})$]  There exist $C_1, C_2 >0 $, and $s >1$ such that $|f(t)| \leq C_1 +C_2 t^s $ for all $t \geq 0$.
\end{itemize}
Note that any $f$ satisfying $(H3)$ will satisfy $(\tilde{H3})$.
\begin{theorem}
\label{regularity}
Let $K: B_1 ^c \rightarrow \mathbb{R}^{+}$, $f:[0,\infty) \rightarrow \mathbb{R}$ be H\"older continuous functions. Assume that $K$ satisfies $(H1)$ and $f$ satisfies $(\tilde{H3})$. Suppose $u \in \D_0^{1,2}(B_1^c)$ is a solution to \eqref{Problem 1}, then $u$ is in $L^\infty(B_1^c)$.
\end{theorem}
We now discuss the positivity of the solution obtained in Theorem \ref{existence} under an additional assumption on the weight function. We assume the following. 
\begin{itemize}
\item[(H5)] There exist $C>0$ and $\gamma > 3$ such that $K(x) \leq \frac{C}{|x|^{\gamma}}$ for $|x| >1$.
\end{itemize}
Under this condition on $K$, we first establish a representation formula for solutions to \eqref{Problem 1}.  
\begin{theorem}
\label{Greens}
Let $K$, $f$ be locally Holder continuous functions, $f$ be as in Theorem \ref{regularity} and $K$ satisfies $(H5)$. Let $u \in \mathcal{D}^{1,2}_0(B^c_1)$ be a solution of \eqref{Problem 1}. Then $u$ has the following representation
\begin{equation}\label{Equation: Greens}
u(x)=\int_{B_1^c}\lambda G(x,y)K(y)f(u(y)) \text{ , } x\in B_1^c,
\end{equation}\\
where $G(x,y)=-\frac{1}{2\pi}\left (\log{|y-x|}-\log\left({|x||y-\frac{x}{|x|^2}|}\right ) \right )$, $x,y\in B_1^c$ is the Green's function.
\end{theorem}
\begin{remark}
Since $G>0$, equation (\ref{Equation: Greens}) implies that the solutions to (\ref{Problem 1}) are nonnegative when $f \geq 0$. 
\end{remark} 
Finally, we state the positivity result. 
\begin{theorem}
\label{positivity}
Let $K: B_1 ^c \rightarrow \mathbb{R}^{+}$, $f:[0,\infty) \rightarrow \mathbb{R}$ be H\"older continuous functions. Assume that $K$ satisfies $(H5)$ and $f$ satisfies $(H2)$-$(H4)$, then the solution obtained in Theorem \ref{existence} is positive.
\end{theorem}
We also state a nonexistence result for large values of the parameter $\lambda$. The proof of this is exactly as in \cite{MR4500097} (see Theorem 1.4) and is therefore omitted. For this result, we only assume the following conditions on $f$.
\begin{enumerate}
\item[(H6)] There exists $\beta >0$ such that $f(\beta) = 0$, $f(x)<0$ for all $x<\beta$ and $f(x) >0$ for all $x>\beta$.
\item[(H7)] $\lim_{s \rightarrow \infty}\frac{f(s)}{s} = \infty$.
\end{enumerate}
\begin{theorem}
Let $K:B_1^c \rightarrow \mathbb{R}^+$ be a continuous function, $f \in C^1[0,\infty)$ satisfies $(H6), (H7)$. Then, there exists $\tilde{\lambda} > 0$ such that for $\lambda \geq \tilde{\lambda}$, \eqref{Problem 1} has no positive solution in $C^2(B_1^c)$.
\end{theorem}

The rest of the paper is organized as follows.  In Section 2, we provide the proof of the continuous and compact embeddings. The existence and regularity results are established in Section 3 and in Section 4, we prove the representation formula for solutions to \eqref{Problem 1} and the positivity result. Throughout the paper, we denote generic positive constants by $C$, which may represent different constants within the same proof.

\section{Continuous and Compact Embeddings of $\D_0^{1,N}(B_1^c)$}
We first establish an inequality for functions in $C_c^\infty(B_1^c)$, the space of all compactly supported smooth functions on $B_1^c$. We use arguments from  \cite{MR768824}, where a similar inequality is proved. We are able to remove some conditions on the parameters given in \cite{MR768824} as our domain is the exterior of the unit ball. A special case of the following Lemma will be used to prove our embeddings. 
\begin{lemma}\label{Lemma:Embedding}
Let $p , q , r, a , \alpha,\beta,\sigma \in \R$ with $p,q \geq 1,$ $r>0$, $0<a<1$ and $\alpha-\sigma\geq0$. Let $$\gamma=a\sigma+(1-a)\beta \text{ , and assume that } \frac{1}{r}+\frac{\gamma}{N}=a \left( \frac{1}{p}+\frac{\alpha-1}{N} \right)+( 1-a ) \left( \frac{1}{q}+\frac{\beta}{N} \right).$$ For $\rho \geq1$, set $R_{\rho}$=$\{x:\rho<|x|\leq 2\rho\}$. If $u\in C_c^\infty(B_1^c)$ and $\delta=p\gamma+\frac{pN}{r}-N$, then 
\begin{equation}
\label{Lemmainequality}
\int_{R\rho}|x|^{\gamma r}|u|^r\leq C \left( \int_{R_\rho}|x|^{\alpha p}|\nabla u|^p \right)^\frac{ar}{p} \left( \int_{R\rho}|x|^{\beta q}|u|^q \right)^\frac{(1-a)r}{q}+C \left( \int_{R\rho}|x|^\delta |u|^p \right)^\frac{r}{p},
\end{equation}
where $C>0$ is a constant independent of $\rho$.
\end{lemma}
\begin{proof}
We first consider the case $\rho=1$, i.e.,  $R_1=\{x:1<|x|\leq 2\}$. Suppose that
\begin{equation} \label{Embedding: Equation 1} 
\frac{1}{m}=\frac{a}{p}+\frac{(1-a)}{q}-\frac{a}{N}>0.
\end{equation}
 Let $\overline{u}=\frac{1}{|R_1|}\int_{R_1}u$, where $|R_1|$ denotes the measure of $R_1$. By using a standard interpolation inequality {(see \cite{MR109940})}, we obtain
\begin{equation}
\label{ii}
\int_{R_1}|u-\overline{u}|^m\leq C \left( \int_{R_1}|\nabla u|^p \right)^\frac{am}{p}\left( \int_{R_1}|u-\overline{u}|^q \right)^\frac{(1-a)m}{q} .
\end{equation}
Since $a>0$, and $\alpha-\sigma \geq 0$, it follows that
\begin{equation*}
\frac{1}{r}=a \left( \frac{1}{p}+\frac{\alpha-\sigma}{N}-\frac{1}{N} \right)+(1-a)\frac{1}{q}
\end{equation*}
\begin{equation*}
\geq a \left( \frac{1}{p}-\frac{1}{N} \right)+(1-a)\frac{1}{q}
=\frac{1}{m},
\end{equation*}
and hence $r\leq m$. Applying Holder's inequality to \eqref{ii}, we get
\begin{equation}\label{Embedding: Equation 5}
\int_{R_1}|u-\overline{u}|^r\leq C \left(\int_{R_1}|u-\overline{u}|^m \right)^\frac{r}{m}\\
\leq C \left( \int_{R_1}|\nabla u|^p \right)^\frac{ar}{p} \left(\int_{R_1}|u-\overline{u}|^q \right)^\frac{(1-a)r}{q} .
\end{equation}
If (\ref{Embedding: Equation 1}) does not hold,  then $\frac{a}{p}+\frac{(1-a)}{q}\leq \frac{a}{N}$, which implies $\frac{1}{q}>\frac{1}{p}-\frac{1}{N}$. Hence, in this case, we obtain
\begin{equation} \label{Embedding: Equation 2}
\int_{R_1}|u-\overline{u}|^r
\leq C \left(\int_{R_1}|\nabla u|^p \right)^\frac{br}{p} \left(\int_{R_1}|u-\overline{u}|^q \right)^\frac{(1-b)r}{q}, 
\end{equation}
where the parameter $b$ is defined as follows: if $r\leq q$, then set $b=0$  and if $r\geq q$ then choose $b$ such that  $\frac{1}{r}=b \left(\frac{1}{p}-\frac{1}
{N} \right)+(1-b)\frac{1}{q}$. 

We now claim that in either case $b \leq a$. If $b=0$, then the inequality is immediate since $a>0$. Assume next that $b\neq 0$ so that $r\geq q$.  From the assumption $\frac{a}{p}+\frac{(1-a)}{q}\leq \frac{a}{N}$ we deduce that $\frac{1}{q}\leq a\left( \frac{1}{N}-\frac{1}{p}+\frac{1}{q} \right)$. Equivalently, $\frac{1}{q \left( \frac{1}{N}-\frac{1}{p}+\frac{1}{q} \right)}\leq a$. On the other hand, by the definition of $b$,
\begin{equation*}
b \left(\frac{1}{q}+\frac{1}{N}-\frac{1}{p} \right)= \frac{1}{q}-\frac{1}{r}
 < \frac{1}{q}.
\end{equation*}
Therefore, $ b<\frac{1}{q \left(\frac{1}{N}-\frac{1}{p}+\frac{1}{q} \right)}\leq a$, which proves our claim. Applying Sobolev's inequality we have,
\begin{equation} \label{Embedding: Equation 3}
\left( \int_{R_1}|u-\overline{u}|^q \right)^\frac{1}{q}\leq C \left( \int_{R_1}|\nabla u|^p \right)^\frac{1}{p}.
\end{equation}
On the other hand, since $b \leq a$, we may rewrite (\ref{Embedding: Equation 2}) as
\begin{equation} \label{Embedding: Equation 4}
\int_{R_1}|u-\overline{u}|^r \leq C \left( \int_{R_1}|\nabla u|^p \right)^\frac{ar}{p} \left(\int_{R_1}|\nabla u|^p \right )^\frac{(b-a)r}{p} \left(\int_{R_1}|u-\overline{u}|^q \right)^\frac{(1-a)r}{q} \left(\int_{R_1}|u-\overline{u}|^q \right)^\frac{(a-b)r}{q}.
\end{equation}
Now, using (\ref{Embedding: Equation 3}), we estimate
\begin{equation}
\left( \int_{R_1}|u-\overline{u}|^q \right )^\frac{(a-b)r}{q}\leq C \left(\int_{R_1}|\nabla u|^p \right)^\frac{(a-b)r}{p} .
\end{equation}
Substituting this into (\ref{Embedding: Equation 4}), we obtain
\begin{equation}
\int_{R_1}|u-\overline{u}|^r\leq C \left( \int_{R_1}|\nabla u|^p \right)^\frac{ar}{p} \left( \int_{R_1}|u-\overline{u}|^q \right)^\frac{(1-a)r}{q} .
\end{equation}
Thus, even in this case, we recover the same inequality as in equation (\ref{Embedding: Equation 5}).
By the change of variable formula, we have
\begin{align*}
\int_{R_\rho}|u|^r =\int_{R_1}|u(\rho x)|^r\rho ^N dx =\rho^N\int_{R_1}|u_1|^r dx ,
\end{align*}
where $u_1=u(\rho x)$. We also note that, $\left( \int_{R_1}|u-\overline{u}|^q \right)^\frac{1}{q}\leq C\left (\int_{R_1}|u|^q \right)^\frac{1}{q}$. Hence, 
\begin{align*}
\left( \int_{R_1}|u|^r \right)^\frac{1}{r} & \leq \left( \int_{R_1}|u-\overline{u}|^r \right)^\frac{1}{r} + \left( \int_{R_1}|\overline{u}|^r \right)^\frac{1}{r}  \\
& \leq C \left( \int_{R_1}|\nabla u|^p \right)^\frac{a}{p} \left(\int_{R_1}|u|^q \right)^\frac{(1-a)}{q}+C  \int_{R_1}|u|  \\
& \leq C \left( \int_{R_1}|\nabla u|^p \right)^\frac{a}{p} \left(\int_{R_1}|u|^q \right)^\frac{(1-a)}{q}+C \left(\int_{R_1}|u|^p \right)^\frac{1}{p}.
\end{align*}
Therefore
\begin{align*}
\left( \int_{R_\rho}|u|^r \right)^\frac{1}{r}
& = \rho^\frac{N}{r} \left( \int_{R_1}|u_1|^r dx \right)^\frac{1}{r} \\
& \leq C\rho^\frac{N}{r} \left( \int_{R_1}|\nabla u_1|^p \right)^\frac{a}{p} \left( \int_{R_1}|u_1|^q \right )^\frac{(1-a)}{q}+ C\rho^\frac{N}{r} \left(\int_{R_1}|u_1|^p \right)^\frac{1}{p} . 
\end{align*}
Multiplying both side by $\rho^\gamma$, we get
\begin{align*}
\begin{split}
\left( \int_{R_\rho}\rho^{\gamma r}|u|^r \right)^\frac{1}{r} & \leq C\rho^{\frac{N}{r}+\gamma} \left(  \int_{R_1}|\nabla u_1|^p \right)^\frac{a}{p} \left(\int_{R_1}|u_1|^q \right)^\frac{(1-a)}{q} + C\rho^{\frac{N}{r}+\gamma} \left( \int_{R_1}|u_1|^p \right)^\frac{1}{p} \\
& =C\rho^{\frac{N}{r}+\gamma} \left( \int_{R_1}|\nabla u(\rho x)|^p\rho^p \right)^\frac{a}{p} \left( \int_{R_1}|u(\rho x)|^q \right)^\frac{(1-a)}{q} \\ & \qquad + C\rho^{\frac{N}{r}+\gamma}\left(\int_{R_1}|u(\rho x)|^p\right)^\frac{1}{p} .
\end{split}
\end{align*}
Now using the fact that, $\gamma+\frac{N}{r}=a(\frac{N}{p}+\alpha-1)+(1-a)(\frac{N}{q}+\beta),$ we obtain
\begin{align*}
\begin{split}
\left( \int_{R_\rho}\rho^{\gamma r}|u|^r \right)^\frac{1}{r} &\leq C\left( \int_{R_1}\rho^{\alpha p}|\nabla u(\rho x)|^p\rho^N \right)^\frac{a}{p}\left( \int_{R_1}\rho^{\beta q}|u(\rho x)|^q \rho^N \right)^\frac{(1-a)}{q} \\ & \qquad + C\left(\int_{R_1}\rho^\delta|u(\rho x)|^p\rho^N\right)^\frac{1}{p} \\
& =C \left( \int_{R_\rho}\rho^{\alpha p}|\nabla u|^p \right)^\frac{a}{p}  \left( \int_{R_\rho}\rho^{\beta q}|u|^q \right)^\frac{(1-a)}{q}+C \left( \int_{R_\rho}\rho^\delta|u|^p \right)^\frac{1}{p} .
\end{split}
\end{align*}
Since $R_{\rho}$=$\{x:\rho<|x|\leq 2\rho\}, \rho>1$, depending on the signs of the exponents $\gamma r$, $\alpha p$, and $\beta q$ 
we can adjust the constants accordingly, and thus obtain
\begin{equation}
\left( \int_{R_\rho}|x|^{\gamma r}|u|^r \right)^\frac{1}{r} \leq  C \left( \int_{R_\rho}|x|^{\alpha p}|\nabla u|^p \right)^\frac{a}{p} \left( \int_{R_\rho}|x|^{\beta q}|u|^q \right)^\frac{(1-a)}{q} + C\left( \int_{R_\rho}|x|^\delta|u|^p \right)^\frac{1}{p} .
\end{equation}
Since $(a+b)^m \leq 2^m (a^m+b^m)$, for $a, b \geq0, m>0$, we get
\begin{equation*}
\int_{R_\rho}|x|^{\gamma r}|u|^r  \leq C \left( \int_{R_\rho}|x|^{\alpha p}|\nabla u|^p \right)^\frac{ar}{p} \left( \int_{R_\rho}|x|^{\beta q}|u|^q \right)^\frac{(1-a)r}{q} + C \left( \int_{R_\rho}|x|^\delta|u|^p \right)^\frac{r}{p},
\end{equation*}
for all $u\in C_c^\infty(B_1^c)$ and $\rho\geq 1$. 
\end{proof}
Using Lemma \ref{Lemma:Embedding}, we now prove Theorem \ref{inequality}.

{\bf {Proof of Theorem \ref{inequality}}:} 
We let $p=q=N,\alpha=\sigma=0 ,\beta=\frac{-(N+\theta)}{N}$ (note that $\beta$ may be taken as any sufficiently large negative number) and $0<a<1$ in the preceeding Lemma.
Then $$\gamma=(1-a)\beta=\frac{-(N+\theta)(1-a)}{N}.$$
Moreover, since $$\frac{1}{r}+\frac{\gamma}{N}=a \left(\frac{1}{p}+\frac{\alpha-1}{N} \right)+(1-a) \left(\frac{1}{q}+\frac{\beta}{N} \right),$$ 
we obtain $r=\frac{N}{1-a}$, so that $N<r<\infty$ and $\delta=-(N+\theta(1-a))$. Thus in this case, the inequality in Lemma \ref{Lemma:Embedding} reduces to 
\begin{equation}
\label{special case ineq}
\int_{R\rho}|x|^{\gamma r}|u|^r\leq C \left( \int_{R_\rho}|\nabla u|^N \right)^\frac{ar}{N} \left( \int_{R\rho}|x|^{-(N+\theta)}|u|^N \right)^\frac{(1-a)r}{N}+C \left( \int_{R\rho}|x|^\delta |u|^N \right)^\frac{r}{N},
\end{equation}
for all $u\in C_c^\infty(B_1^c)$. We now recall the following inequalities, valid for any sequence $x_k \geq 0$, $y_k \geq 0$ and exponents $c,d \geq0 $.
\begin{equation} \label{Embedding: Equation 7}
\sum x_k^c y_k^d\leq \left (\sum x_k \right )^c \left( \sum y_k \right)^d, \text{ if } c+d\geq 1 ,
\end{equation}
and
\begin{equation} \label{Embedding: Equation 8}
\sum x_k^c\leq \left(\sum x_k \right)^c, \text{ if } c\geq 1.
\end{equation}
Since, in our case $\frac{ar}{p}+\frac{(1-a)r}{q}=\frac{r}{N}>1$, and (\ref{special case ineq}) holds for all $\rho\geq 1$, adding the estimates (\ref{special case ineq}) over all annuli and applying the monotone convergence theorem together with (\ref{Embedding: Equation 7})-(\ref{Embedding: Equation 8}), we obtain 
\begin{equation}
 \int_{B_1^c}|x|^{\gamma r}|u|^r  \leq C \left( \int_{B_1^c}|\nabla u|^N \right)^\frac{ar}{N} \left( \int_{B_1^c}|x|^{-(N+\theta)}|u|^N \right)^\frac{(1-a)r}{N} + C \left( \int_{B_1^c}|x|^\delta|u|^N \right )^\frac{r}{N},
\end{equation}
for all $u\in C_c^\infty(B_1^c)$ and $N<r<\infty$. Applying the inequality $(a+b)^m \leq 2^m(a^m+b^m)$, yields 
\begin{equation}
\label{II}
\left( \int_{B_1^c}|x|^{\gamma r}|u|^r \right )^\frac{1}{r}  \leq C  \left(  \int_{B_1^c}|\nabla u|^N \right )^\frac{a}{N} \left( \int_{B_1^c}|x|^{-(N+\theta)}|u|^N \right)  ^\frac{(1-a)}{N} + C \left(  \int_{B_1^c}|x|^\delta|u|^N \right)^\frac{1}{N},
\end{equation}
for all $u\in C_c^\infty(B_1^c)$ and $N<r<\infty$. Note that from our choice of the parameters, $\gamma r=-(N+\theta)$.
\qed

We now recall an embedding of $\D_0^{1,N}(B_1^c)$ from \cite{MR3347486}. 
\begin{lemma}
\label{lemma 2.1.}\label{Compact Embedding:Anoop}
(See \cite{MR3347486}) $\D_0^{1,N}(B_1^c) $  is compactly embedded into the weighted Lebesgue space $L^N(W(|x|),B_1^c)$ for $W\in L^1((1,\infty),[r\log r]^{N-1})$.
\end{lemma}
It is easy to verify that our weight $K \in L^1((1,\infty),[r\log r]^{N-1})$. Using Lemma \ref{Compact Embedding:Anoop} and Theorem \ref{inequality}, we now prove Theorem \ref{embedding}. \\

{\bf{Proof of Theorem \ref{embedding}}:} 
Since $C_c^\infty(B_1^c)$ is dense in $\mathcal{D}^{1,N}_0(B^c_1)$, the inequality \eqref{II} in Theorem \ref{inequality} extends to all $u \in \D_0^{1,N}(B_1^c)$, thereby establishing the continuity of the embedding of $\mathcal{D}^{1,N}_0(B^c_1)$ into $L^r(B_1^c;K(x))$, for $r\in(N,\infty)$. The continuous embedding of $\mathcal{D}^{1,N}_0(B^c_1)$ into $L^N(B_1^c;K(x))$ follows by Lemma \ref{Compact Embedding:Anoop}. Now, for $1\leq r < N$, we have
\begin{align*}
\int_{B_1^c} K(x)|u|^r &= \int_{B_1^c} K(x)^{\frac{N-r}{N}} K(x)^\frac{r}{N} |u|^r \\
& \leq \left(\int_{B_1^c} K(x) |u|^N \right)^{\frac{r}{N}} \left (\int_{B_1^c} K(x)\right)^\frac{N-r}{N},
\end{align*}
where the above inequality is obtained by applying H\"older's inequality with exponents $\frac{N}{r}$ and $\frac{N}{N-r}$. Hence, it follows that $\mathcal{D}^{1,N}_0(B^c_1)$ is continuous embedded into $L^r(B_1^c;K(x))$, for $1 \leq r < N$. 

To prove compactness, let $\{u_m\}$ be a bounded sequence in $\mathcal{D}^{1,N}_0(B^c_1)$. Then there exists $u\in \mathcal{D}^{1,N}_0(B^c_1)$ such that $u_m\rightharpoonup u$ in $\mathcal{D}^{1,N}_0(B^c_1)$. Moreover, there exists $M>0$ such that 
\begin{equation}
\left( \int_{B_1^c}|\nabla u_m-\nabla u|^N \right) ^\frac{1}{N}\leq M , \text{ for all } m.
\end{equation}
By compactness of the embeddings in Lemma \ref{Compact Embedding:Anoop}, there exists a subsequence $\{u_{m_k}\}$ of $\{u_m\}$ such that
\begin{equation}
\left( \int_{B_1^c}|x|^{-(N+\theta)}|u_{m_k}-u|^N \right)^\frac{1}{N} \rightarrow 0  \text{ , and } \left( \int_{B_1^c}|x|^\delta|u_{m_k}-u|^N \right)^\frac{1}{N} \rightarrow 0 ,
\end{equation}
as $m_k\rightarrow \infty$. Therefore, from Theorem \ref{inequality}, we have
\begin{align*}
\begin{split}
\left( \int_{B_1^c}|x|^{-(N+\theta)}|u_{m_k}-u|^r \right)^\frac{1}{r} & \leq C \left( \int_{B_1^c}|\nabla u_{m_k}-\nabla u|^N \right)^\frac{a}{N} \left( \int_{B_1^c}|x|^{-(N+\theta)}|u_{m_k}-u|^N \right) ^\frac{(1-a)}{N} \\ &\qquad+ C \left( \int_{B_1^c}|x|^\delta|u_{m_k}-u|^N \right) ^\frac{1}{N} \\
& \leq C M^a \left(\int_{B_1^c}|x|^{-(N+\theta)}|u_{m_k}-u|^N \right)^\frac{(1-a)}{N} \\
& \qquad +C\left(\int_{B_1^c}|x|^\delta|u_{m_k}-u|^N\right)^\frac{1}{N} \\ & \rightarrow 0 \text{, as } m_k\rightarrow \infty.
\end{split}
\end{align*}
This proves the compactness of the embedding $\mathcal{D}^{1,N}_0(B^c_1)$  into $L^r(B_1^c;K(x))$, for all $r\in(N,\infty)$. Similar arguments also provide the compactness of the embedding of $\mathcal{D}^{1,N}_0(B^c_1)$  into $L^r(B_1^c;K(x))$, for all $1\leq r \leq N$.
\qed
\section{Existence and regularity of solution to \eqref{Problem 1}}
In this section, we establish the existence and regularity of solutions to \ref{Problem 1}. We first recall the Mountain Pass Lemma which will be used to establish the existence. 
\begin{lemma}
(See \cite{rabinowitz1986minimax}) (Mountain Pass Lemma). \label{lem3}
Let $E$ be a real Banach space and $J \in C^1(E,\mathbb{R})$ satisfies the Palais-Smale condition ($PS$). Suppose also that
\begin{enumerate}
\item[(a)]  $J(0) = 0$,
\item[(b)] there exist constants $\rho, \alpha > 0$ such that $J(v) \geq \alpha$ for all $v \in E$ with $\| v\|_E = \rho$,			\item[(c)] there exists an $e \in E $ with $\| e\|_E > \rho$ such that $J(e) \leq 0$.
\end{enumerate}
Then $J$ possesses a critical value $c\geq \alpha$ which can be characterized as $$c = \inf _{\gamma \in \Gamma} \max _{t \in [0,1]} J(\gamma(t)),$$ where $$\Gamma= \{ \gamma \in C([0,1],E) : \gamma(0) =0 , \gamma(1)=e \}.$$
\end{lemma}
We provide below the proof of Theorem \ref{existence}. The embedding that we established in Theorem \ref{embedding} will be used throughout the proof. 

{\bf {Proof of Theorem \ref{existence} }:}
Define the functional $J : \D_0^{1,2}(B_1^c) \rightarrow \mathbb{R}$ by $J(u) = \frac{1}{2} \| \nabla u \|_2 ^2 - \lambda \tilde{F}(u),$ where $\tilde{F}: \D_0^{1,2}(B_1^c) \rightarrow \mathbb{R}$ is the functional $\tilde{F}(u) =  \int_{B_1^c} K(x) F(u) dx $  and $F(u) = \int_0 ^ u f(t) dt$. We now prove that $J$ satisfies the Mountain Pass Lemma. 
\\ 
\\
\textit{Step 1: $J$ is well defined. } 

We will prove that the second part of $J$ is finite for all $u \in \D_0^{1,2}(B_1^c)$. Indeed, we estimate
\begin{align*}
|\int_{B_1^c} K(x) F(u) dx| &\leq |f(0)| \int_{B_1 ^c} K(x) |u| dx + \frac{B}{s+1} \int_{B_1 ^c} K(x) |u| ^{s+1} dx .
\end{align*}
By Theorem \ref{Theorem:Compact embedding}, the embeddings yield 
\begin{align*}
\int_{B_1 ^c} K(x) |u| ^{s+1} dx \leq C \| \nabla u \|_2^{s+1},
\end{align*}
and 
$$\int_{B_1^c} K(x) |u| dx 
\leq C \| \nabla u \|_2 .$$
Therefore,
\begin{equation}
\left| \int_{B_1^c} F(u) K(x) dx \right| \leq C \left( \| \nabla u \|_2 + \| \nabla u \|_2^{s+1} \right).
\end{equation}
This shows that $J$ is well defined. Using our embeddings and arguments analogous to those in \cite{MR4500097} and  \cite {rabinowitz1986minimax}, one easily sees that $J$ is $C^1$ and $J'(u): \D_0^{1,2}(B_1^c) \rightarrow \mathbb{R}$, given by $$J'(u)(\phi) = \int_{B_1^c} \nabla u \cdot \nabla \phi - \lambda  \tilde{F}'(u)(\phi) \text{, where } \tilde{F}'(u)(\phi) = \int_ {B_1 ^c} K(x) f(u) \phi , $$ is continuous for each $u\in \D_0^{1,2}(B_1^c)$ and the map $\tilde{F}': \D_0^{1,2}(B_1^c) \to (\D_0^{1,2}(B_1^c))'$ is compact. \\
\\
\textit{Step 2: $J$ satisfies the Palais–Smale condition.} 

Let $u_m$ be a sequence in $\D_0^{1,2}(B_1^c)$ such that $J(u_m)$ is bounded and $J'(u_m) \rightarrow 0$. To prove that $u_m$ has a convergent subsequence, we begin by showing that $u_m$ is a  bounded sequence. Since $u_m = u_m ^{+} - u_m ^{-} $, and $u_m ^+$ and $u_m ^- $ are in $\D_0^{1,2}(B_1^c)$, it is enough to show that both $u_m ^{+}$ and $u_m^{-}$ are bounded sequences. For the negative part, testing with $u_m^-$ gives
\begin{align*}
J'(u_m)(u_m ^-) &= - \int_{B_1 ^c} | \nabla u_m ^- |^2 dx - \lambda \int_{B_1 ^c} K(x) f(u_m) u_m ^- dx \\
&= -\| \nabla u_m ^- \|_2^2 - \lambda \int_{B_1 ^c} K(x) f(0) u_m ^- dx \\
&\leq  -\| \nabla u_m ^- \|_2^2 +  C \lambda  |f(0)| \| \nabla u_m^- \|_2 .
\end{align*} 
Hence,
\begin{equation}\label{Equation 1}
\| \nabla u_m ^- \|_2 ^2 \leq | J'(u_m)(u_m ^ -)| + C \lambda |f(0)| \| \nabla u_m^- \|_2 ,
\end{equation} 
from which it follows that $u_m ^-$ is bounded in $\D_0^{1,2}(B_1^c)$. Boundedness of $J(u_m)$ implies the existence a constant $C > 0$ such that
\begin{equation} \label{Equation 2}
\frac{1}{2} \| \nabla u_m \|_2 ^2 - \lambda \int_{B_1 ^c} K(x) G(u_m) dx - \lambda \int_{B_1 ^c} K(x) f(0) u_m dx \leq C .
\end{equation}
Since $u_m^-$ is bounded,
\begin{equation}\label{Equation 3}
\int_{B_1^c} K(x) f(0) u_m dx \leq \int_{B_1 ^c} K(x) |f(0)| u_m ^ - dx \leq C.
\end{equation}
Moreover, we have $G(u_m) \leq G(u_m ^+)$ and $ \| \nabla u_m ^+ \|_2 ^2 \leq \|  \nabla u_m \|_2 ^2$ for all $m$. Substituting this along with (\ref{Equation 3}) into (\ref{Equation 2}) gives
\begin{equation}\label{Equation 4}
\frac{1}{2} \| \nabla u_m ^+ \|_2 ^2 - \lambda \int_{B_1 ^c} K(x) G(u_m ^+) dx \leq C .
\end{equation} 
By assumption ($H4$), we have $ - G(u_m ^ +) \geq - \frac{u_m ^+}{\mu} g(u_m ^+).$ Hence,
\begin{equation}\label{Equation 5}
\frac{1}{2} \| \nabla u_m ^+ \|_2 ^2 - \frac{\lambda}{\mu} \int_{B_1 ^c} K(x) u_m ^+ g(u_m ^+) dx \leq C.
\end{equation}
On the other hand, 
\begin{equation}\label{Equation 6}
J'(u_m)(u_m ^+) = \| \nabla u_m ^+ \|_2 ^2 - \lambda \int_{B_1 ^c} K(x) g(u_m ^+) u_m ^+  dx - \lambda \int_{B_1 ^c} K(x) f(0) u_m ^+ dx .
\end{equation}
Combining the equations (\ref{Equation 5}) and (\ref{Equation 6}), we obtain
\begin{align}
C 
&\geq (\frac{1}{2} - \frac{1}{\mu}) \| \nabla u_m ^+ \|_2 ^2 + \frac{1}{\mu} J'(u_m) (u_m ^ +) + C \frac{\lambda}{\mu} f(0)  \| \nabla u_m ^+ \|_2 . \label{Equation 7}
\end{align}
As $\mu > 2$, from (\ref{Equation 7}) it follows that $u_m ^+$ is bounded in $\D_0^{1,2}(B_1^c).$ Therefore, $u_m$ is bounded. Since $\D_0^{1,2}(B_1^c)$ is reflexive, $u_m $ has a subsequence, again denoted by $u_m$, such that $u_m \rightharpoonup u$ for some $u \in \D_0^{1,2}(B_1^c)$. The strong convergence $u_m \rightarrow u$ then follows as in \cite{MR4500097}, by using the facts that $J'(u_{m}) \to 0$ and that the map $\tilde{F}'$ is compact. Hence, $J$ satisfies the Palais–Smale condition.\\
\\
\textit{Step 3: The Mountain-Pass geometry of $J$.} 

Clearly $J(0)=0$. Let $\rho =1$ and $u \in \D_0^{1,2}(B_1^c)$ be such that $\| \nabla u \|_2 = 1$. Then
\begin{equation}\label{Equation 11}
J(u) = \frac{1}{2} - \lambda \int_{B_1 ^c} K(x) \left[ G(u) +f(0)u \right] dx.
\end{equation}
We estimate the integrals on the right hand side of (\ref{Equation 11}) as follows. We have
\begin{align*}
\int_{B_1 ^c} K(x) G(u) dx  &\leq \frac{B}{s+1} \int_{B_1^c} K(x) |u|^{s+1} dx \\
&\leq C \| \nabla u \|_2^{s+1},
\end{align*}
and
\begin{equation*}
\int_{B_1 ^c} K(x) f(0)u \leq C |f(0)| \| \nabla u \|_2.
\end{equation*}
Hence
\begin{equation}\label{Equation 9}
J(u) \geq \frac{1}{2} - \lambda C  \| \nabla u \|_2 ^{s+1} - \lambda C |f(0)| \| \nabla u \|_2  .
\end{equation}
Thus, there exists $\lambda_0 >0 $ such that, if $\lambda \leq \lambda_0$, $J(u) > \frac{1}{2}$.
	
To verify the third condition of the Mountain Pass Lemma, let $w >0$ be in $C_c^{\infty} (B_1 ^c)$, $t>0$, and consider
\begin{align*}
J(tw) &= \frac{t^2}{2} \| \nabla w \|_2^2 - \lambda \int_{B_1 ^c} F(tw) K(x) \\
&= \frac{t^2}{2} \| \nabla w \|_2^2- \lambda \int_{\text{supp}(w)} \left[G(tw)+f(0)tw \right] K(x) .
\end{align*} 
By (H3), we have $G(t) \geq \frac{A}{s+1} t^\mu$. Substituting into the above expression, we get
\begin{equation}\label{Equation 10}
J(tw) \leq \frac{t^2}{2} \| \nabla w \|_2^2 - \lambda \frac{A}{s+1} t ^ \mu \int_{\text{supp}(w)} K(x) w^ \mu + \lambda |f(0)| t\int_{\text{supp}(w)} K(x) w . 
\end{equation}
Since $\mu > 2$, the leading term on the right-hand side is negative and grows faster than the others as $t \rightarrow \infty$. Hence, from (\ref{Equation 10}), we conclude that $J(tw) \rightarrow - \infty$ as $t \rightarrow \infty$. In particular, we can choose  $t>0$ large enough so that $J(tw) <0$. Thus, all the conditions of the Mountain Pass Lemma are satisfied and the functional $J$ has a critical point, say $u$. By standard regularity results it follows that $u \in C^2(B_1^c) \cap C_{loc}^{1,\alpha}(\overbar{B_1^c})$ (see \cite{MR737190}, \cite{Tolksdorf-84}).

\qed

Next, we proceed to prove boundedness of the solution stated in Theorem \ref{regularity}. For this, we recall the following result from \cite{MR3925556}.
\begin{lemma} 
\label{lemma 2.2.}The Kelvin transform $\mathcal{K}$ defined by $\hat{u}(x)=(\mathcal{K}u)(x)=u(\frac{x}{|x|^2})$ provides an order-preserving, isometric isomorphism from $H=H_0^1(B_1)$ to  $\mathcal{D}^{1,2}_0(B^c_1)$. 
\end{lemma}

Using lemma \ref{lemma 2.2.}, one can see that  if $u \in \mathcal{D}^{1,2}_0(B^c_1)$ is a weak solution of \eqref{Problem 1}, then $\hat u$ is a weak solution to the problem,
\begin{equation}
\label{ballproblem}
\begin{cases}
-\Delta \hat{u} =\lambda \hat{K}(x)f(\hat{u})\frac{1}{|x|^4} \mbox{ in } B_1, \\
\hat{u} =0 \hspace{2.5cm}  \mbox{ on } \partial B_1,
\end{cases}
\end{equation}
where $B_1$ is the open unit ball in $\mathbb{R}^2$, and $\hat{K}(x)=K(\frac{x}{|x|^2})$ (see \cite{MR3925556} for details). We note that for the class of weight functions that we consider, this problem may have a singularity at the origin. 

{\bf {Proof of Theorem \ref{regularity}} :} Let $u \in \mathcal{D}_0^{1, 2}(B_1^c)$ be a solution to \eqref{Problem 1}. Then, as mentioned above $\hat{u}$ is a weak solution to the problem \eqref{ballproblem}. Define $\zeta(x)=\hat{K}(x) f(\hat{u})\frac{1}{|x|^4}$ and note that if $\zeta\in L^\eta(B_1)$ for some $\eta>1$, then standard elliptic regularity results (see \cite{MR2777537},\cite{MR737190}) imply that $\hat{u}\in C^\alpha(\overbar B_1)$. Since $K(x)\leq \frac{1}{|x|^{2+\theta}}$ for some $\theta>0$ for all $x \in B_1^c,$ $\hat{K}(x) \leq |x|^{2+\theta}$ for all $x \in B_1.$ To show that for each $\theta>0$ there exists $\eta>1$ such that $\zeta\in L^\eta(B_1)$, we use the Kelvin transform and estimate the integral as follows.
\begin{align*}
 \int_{B_1} |\zeta|^\eta dx &\leq \int_{B_1}\frac{|x|^{(2+\theta)\eta}}{|x|^{4\eta}}|f(\hat{u})|^\eta  dx \\
 &= \int_{B_1^c}\frac{|x|^{4\eta}}{|x|^{(2+\theta)\eta}}|f(u)|^\eta \frac{1}{|x|^4} dx \\
 &=\int_{B_1^c}\frac{1}{|x|^{4+(\theta-2)\eta}}|f(u)|^\eta dx \\
 &\leq \tilde C_1 \int_{B_1^c}\frac{1}{|x|^{4+(\theta-2)\eta}} dx +\tilde C_2 \int_{B_1^c}\frac{1}{|x|^{4+(\theta-2)\eta}}|u|^{s\eta} dx .
\end{align*}
Since $\theta>0$, we have $4+(\theta-2)\eta>4-2\eta$. Consequently, there exists $\eta>1$ such that $4+(\theta-2)\eta>2$, by the continuity of the function $4+(\theta-2)\eta$ with respect to $\eta$. Thus, by our embedding result, it follows that the above integral is finite. Hence, $\hat{u}\in C^\alpha(\overbar B_1)$. In particular $\lim_{x\rightarrow 0} \hat{u}(x)$ exists, which implies $\lim_{x\rightarrow \infty}{u}(x)$ also exists. This shows that $u\in C^\alpha(\overbar B_1^c)\cap L^{\infty}(B_1^c)$.
\qed
\section{Positivity of the solution}
In this section, we first prove Theorem \ref{Greens} which provides a representation formula for solutions to \eqref{Problem 1}.

{\bf  Proof of Theorem \ref{Greens} :} Since $u \in \mathcal{D}^{1,2}_0(B^c_1)$ is a  solution of \eqref{Problem 1}, it's Kelvin transform $\hat{u}$ is a solution of \eqref{ballproblem}. Because $K$ satisfies $(H5)$ and $f$ satisfies $(H3)$, following the same arguments as in the proof of Theorem \ref{regularity}, there exists $\eta>2$ such that $\zeta(x)=\hat{K}(x) f(\hat{u})\frac{1}{|x|^4} \in L^\eta(B_1)$. By standard elliptic regularity results (see \cite{MR737190},\cite{MR2777537}), this implies that $\hat{u} \in C^{1, \alpha}(\overline{B_1})$ for some $\alpha \in (0, 1)$. In particular, there exists a constant $M>0$ such that $|\nabla\hat{u}|\leq M$. 

Now let $x\in B_1 \setminus \{0\}$. Choose $r>0$ small enough so that $r<|x|<1$, and let $\phi(x) = \frac{-1}{2\pi} \log|x|$, the fundamental solution of Laplace's equation. Applying the divergence theorem to $\hat{u}(y) \in C^2(B_1\setminus \{0\}) \cap C^{1, \alpha}(\overline{B_1})$ and $\phi(y-x)$  in the annulus $r<|x|<1$, and following the derivation of Green’s function in Chapter 2 of \cite{MR1625845}, we obtain
\begin{align}
\label{1}
\begin{split}
\hat{u}(x)&=\int_{\partial B_1} \left[ \phi(y-x)\frac{\partial \hat{u}(y)}{\partial \nu}-\hat{u}(y)\frac{\partial \phi(y-x)} {\partial \nu} \right] ds(y) \\  & \quad  + \int_{\partial B_r} \left[ \phi(y-x)\frac{\partial \hat{u}(y)}{\partial \nu} -\hat{u}(y)\frac{\partial \phi(y-x)} {\partial \nu} \right] ds(y) - \int_{B_1 \setminus B_r}\phi(y-x)\Delta \hat{u}(y) dy.  
\end{split}
\end{align}
Note that, we have 
$$\left|\frac{\partial \phi(y-x)} {\partial \nu} \right| \leq \frac{1}{2\pi}\frac{1}{|y-x|}\leq \frac{1}{2\pi}\frac{1}{|x|-|y|}=\frac{1}{2\pi}\frac{1}{|x|-r} , $$ 
and 
$$|\phi(y-x)|=\frac{1}{2\pi}|\log(|y-x|)|\leq \frac{1}{2\pi}|\log(|x| +|y|)|=\frac{1}{2\pi}\log (\frac{1}{|x|+|y|})=\frac{1}{2\pi}\log (\frac{1}{|x|+r}) .$$ 
Since $\hat{u}\in C^{1,\alpha}(\overline{B_1})$, there exist constants $M, N>0$ such that $|\hat{u}|\leq M$ and $\left| \frac{\partial \hat{u}(y)}{\partial \nu} \right| \leq N$. Therefore, we have the following estimates:
\begin{equation}
\label{2}
\left| \int_{\partial B_r}\phi(y-x)\frac{\partial \hat{u}(y)}{\partial \nu} ds(y) \right| \leq \frac{N}{2\pi}\int_{\partial B_r}\log\left( \frac{1}{|x|+r} \right) ds(y)=N  r \log \left(\frac{1}{|x|+r} \right),
\end{equation}
and
\begin{equation}
\label{3}
\left| \int_{\partial B_r}\hat{u}(y)\frac{\partial \phi(y-x)} {\partial \nu} ds(y) \right| \leq \frac{M}{2\pi} \int_{\partial B_r}\frac{1}{|x|-r}=M  \frac{r}{|x|-r}.
\end{equation}
Letting $r\rightarrow 0$ in \eqref{1}-\eqref{3}, we get
\begin{equation}
\label{P1}
\hat{u}(x)=\int_{\partial B_1} \left[ \phi(y-x)\frac{\partial \hat{u}(y)}{\partial \nu}-\hat{u}(y)\frac{\partial \phi(y-x)} {\partial \nu} \right] ds(y)-\int_{B_1}\phi(y-x)\Delta \hat{u}(y) dy.
\end{equation}
Next, for a fixed $x\in B_1 \setminus \{0\}$, define the corrector function $$h^x(y)=\frac{-1}{2\pi}\log \left (|x|\left |y-\frac{x}{|x|^2}\right | \right ) \mbox{ , } y\in B_1.$$ Clearly, $-\Delta h^x(y)=0$ in $B_1$ and $h^x(y)=\frac{-1}{2\pi}\log(|y-x|)$ on $\partial B_1$. Applying the divergence theorem to $h^x(y)$ and $\hat{u}(y)$, we get
\begin{align}
\label{4}
\begin{split}
-\int_{B_1\setminus B_r}h^x(y)\Delta \hat{u}(y) dy &= \int_{\partial B_1} \left[ \hat{u}(y)\frac{\partial h^x(y)}{\partial \nu}-h^x(y)\frac{\partial \hat{u}(y)}{\partial \nu} \right] ds(y) \\ & \quad +\int_{\partial B_r} \left[ \hat{u}(y)\frac{\partial h^x(y)}{\partial \nu}-h^x(y)\frac{\partial \hat{u}(y)}{\partial \nu} \right]ds(y).  
\end{split}
\end{align}
Proceeding as before and letting $r\rightarrow 0$ in \eqref{4}, we obtain
\begin{equation}
\label{P2}
-\int_{B_1}h^x(y)\Delta \hat{u}(y) dy=\int_{\partial B_1} \left[ \hat{u}(y)\frac{\partial h^x(y)}{\partial \nu}-h^x(y)\frac{\partial \hat{u}(y)}{\partial \nu} \right] ds(y) .
\end{equation}
Adding \eqref{P1} and \eqref{P2}, we get
\begin{align}
\label{P3}
\begin{split}
\hat{u}(x)&=\int_{\partial B_1}\left[ \left( \phi(y-x)-h^x(y) \right) \frac{\partial \hat{u}(y)}{\partial \nu} - \hat{u}(y)\left(\frac{\partial \phi(y-x)} {\partial \nu}-\frac{\partial h^x(y)}{\partial \nu} \right) \right] ds(y)\\ & \quad -\int_{B_1}(\phi(y-x)-h^x(y))\Delta \hat{u}(y) dy.
\end{split}
\end{align}
Using the facts that $\hat{u}(y)=0$ and $h^x(y)=\phi(y-x)$ on $\partial B_1$, and defining $$G(x,y)=\phi(y-x)-h^x(y) \mbox{ , } x,y\in B_1 \mbox{ , } x\neq y , $$ we finally obtain
\begin{equation}
\hat{u}(x)=-\int_{B_1}G(x,y)\Delta \hat{u}(y) dy=\int_{B_1}G(x,y)\lambda \hat{K}(y)f(\hat{u}(y))\frac{1}{|y|^4} dy.
\end{equation}
Applying the Kelvin transform to $\hat{u}$, we then get
\begin{equation}
u(x)=\int_{B_1^c}\lambda G(\frac{x}{|x|^2},\frac{y}{|y|^2})K(y
)f(u(y)) dy \mbox{ , } x\in B_1^c.
\end{equation}
A straightforward calculation shows that $G(\frac{x}{|x|^2},\frac{y}{|y|^2})=G(x,y)$, for all $x,y\in B_1^c$.
\qed

\begin{remark}
Since
$ |\nabla \hat{u}(x)|=\frac{1}{|x|^2}|\nabla u(\frac{x}{|x|^2})|$, for all $ x \in B_1 \setminus \{0\} $ and $|\nabla \hat{u}(x)|$ is bounded, it follows that $|\nabla u(x)|\leq\frac{M} {|x|^2}$, for all $x\in B_1^c $. Thus, we observe that solutions have a decaying gradient.     
\end{remark} 
Finally, we are ready to establish the positivity of the solution obtained via the Mountain Pass Lemma.
To achieve this, we first derive suitable estimates for the norm of the solution.	
\begin{lemma}\label{lem7}
Let $u$ be the Mountain Pass solution of the problem (\ref{Problem 1}). Then there exist constants $K_1$, $K_2>0$ and $\overbar{\lambda}>0$ such that $K_1 \lambda^{-\frac{1}{s-1}} \leq  \| \nabla u \|_2 \leq K_2 \lambda^{-\frac{1}{s-1}}$ for $ \lambda \leq \overbar{\lambda}$.
\end{lemma}
\begin{proof}
Note that
\begin{align*}
\| \nabla u \|_2^2 &= \lambda \int_{B_1 ^c} K(x)[g(u)+f(0)] u \\
& \leq \lambda C \left \{\|\nabla u\|_2 ^{s+1} + \lambda \|\nabla u\|_2\right \}.
\end{align*}
Now, division by $\|\nabla u\|_2^2$ and a  rearrangement of the terms give $\|\nabla u\|_2 ^{s-1} \geq C \left[ \lambda^{-1} - \frac{1}{\|\nabla u\|_2} \right] .$ We find $\|\nabla u\|_2$ is bounded from below as $J(u)$ is bounded from below. Therefore, for $\lambda>0$ small enough, there exists $K_1>0$ such that $\| \nabla u \|_2 \geq K_1 \lambda^{\frac{-1}{s-1}}$. Now, consider $w \in C_c^\infty(B_1^c)$, and $t>1$. It can be shown similarly as in \cite{MR4500097} that
\begin{align*}
\frac{d}{dt}(J(tw)) 
& \leq ta - \lambda t^s b + \lambda c ,
\end{align*}
for some constants $a$, $b$, and $c.$ Choosing $t \geq C \lambda^{\frac{-1}{s-1}}$ and noting $s>1$, it follows that $\frac{dJ}{dt} < 0$ for $\lambda$ sufficiently small. From equation (\ref{Equation 10}), we come to the conclusion that $J(tw) \leq J(C \lambda^{\frac{-1}{s-1}} w) \leq  \lambda^{\frac{-2}{s-1}} \|\nabla w\|_2^2$. Letting $\|\nabla w\|_2 =1$, we can see that $J(tw) \leq C \lambda^{\frac{-2}{s-1}}$. Using the characterization for Mountain-Pass critical point, we get
\begin{equation} \label{eqn33}
J(u) \leq C \lambda^{\frac{-2}{s-1}}.
\end{equation}
Now since $u$ is weak solution we have
\begin{equation}
\|\nabla u \|_2^2=\lambda \int_{B_1^c}K(x)f(0)u+\lambda\int_{B_1^c}K(x)ug(u)
\end{equation}
We also have by (H4)
\begin{equation*} 
\begin{aligned}
    J(u) &= \frac{1}{2} \|\nabla u \|_2^2-\lambda \int_{B_1^c}K(x)f(0)u-\lambda\int_{B_1^c}K(x)G(u) \\
    &\geq \frac{1}{2} \|\nabla u \|_2^2-\lambda \int_{B_1^c}K(x)f(0)u-\frac{\lambda}{\mu}\int_{B_1^c}K(x)ug(u).
\end{aligned}
\end{equation*}
Hence, \begin{equation*} \label{eqn34}
    \begin{aligned}
        J(u) &\geq \frac{1}{2} \|\nabla u \|_2^2-\lambda \int_{B_1^c}K(x)f(0)u-\frac{1}{\mu}\|\nabla u \|_2^2 +\frac{\lambda}{\mu}\int_{B_1^c}K(x)f(0)u\\
        &=\left (\frac{1}{2}-\frac{1}{\mu}\right )\|\nabla u \|_2^2-\lambda f(0)\left (1- \frac{1}{\mu}\right )\int_{B_1^c}K(x)u \\
        &\geq \left (\frac{1}{2}-\frac{1}{\mu}\right )\|\nabla u \|_2^2-C\lambda |f(0)|\left (1- \frac{1}{\mu}\right )\|\nabla u \|_2\\
        &=\|\nabla u \|_2 \left [\left (\frac{1}{2}-\frac{1}{\mu}\right )\|\nabla u \|_2-C\lambda |f(0)|\left (1- \frac{1}{\mu}\right )\right ].
    \end{aligned}
\end{equation*}
 Since $\mu>2$, from the above inequality and (\ref{eqn33}), it follows that for $\lambda>0$ small enough, there exists $K_2>0$ such that $\|\nabla u\|_2 \leq K_2 \lambda^{\frac{-1}{s-1}}$.
\end{proof}
\begin{lemma}
Define $S = \{ x \in B_1 ^c : u(x) \geq \epsilon \lambda^{-\frac{1}{s-1}} \}.$ Then there exists a  $\tilde{C}>0$ and $\overbar{\lambda}>0$ such that for all $\lambda \leq \overbar{\lambda},$ $|S| \geq \tilde{C}.$
\end{lemma}
\begin{proof}
Using Lemma \ref{lem7}, we have 
\begin{align*}
\begin{split}
K_1\lambda^{-\frac{2}{s-1}} \leq \| \nabla u \|_2^2 
&= \lambda \int_S K(x) g(u) u + \lambda \int_S K(x)f(0)u \\&\qquad  + \lambda \int_{B_1^c \setminus S} K(x) g(u) u + \lambda \int_{B_1^c \setminus S} K(x) f(0) u \\
& \leq C [ \lambda \| K \|_{L^2(S)} \| \nabla u \|_2 ^{s+1} + \lambda |f(0)|  \| \nabla u \|_2 \\ &\qquad + \lambda \epsilon^{s+1} \int_{B_1 ^c \setminus S} K(x)\lambda ^{-\frac{s+1}{s-1}}  + \lambda |f(0)| \epsilon \int_{B_1 ^c \setminus S} K(x)\lambda^{\frac{-1}{s-1}} ] \\
& \leq C [ \lambda^{\frac{-2}{s-1}}\| K \|_{L^2(S)}   + \lambda^{\frac{s-2}{s-1}}  + \epsilon ^{s+1} \lambda^{\frac{-2}{s-1}} \| K\|_{L^{1}(B_1 ^c \setminus S)} \\ &\qquad +  \epsilon \lambda^{\frac{s-2}{s-1}} \| K \|_{L^{1}(B_1 ^c \setminus S)} ] .
\end{split}
\end{align*}
Now, dividing by $\lambda ^{\frac{-2}{s-1}}$, the above inequality becomes
\begin{equation}
K_1 \leq C [ \| K \|_{L^2(S)} + \lambda^{\frac{s}{s-1}}  + \epsilon ^{s+1}  \| K\|_{L^{1}(B_1 ^c \setminus S)} +  \epsilon \lambda^{\frac{s}{s-1}} \| K \|_{L^{1}(B_1 ^c \setminus S)}] .
\end{equation} 
We can choose $\epsilon >0$ and $\lambda >0$ sufficiently small so as to get $ \| K \|_{L^2(S)}\geq \tilde{C} >0$. Therefore, there exists a $\tilde{C}>0$ and $\overbar{\lambda}>0$ such that for all $\lambda \leq \overbar{\lambda},$ $|S| \geq \tilde{C}.$
\end{proof}
	
Consider $R>1$ so that $|S \cap B(0,R)| >C_1>0$. We will now prove that $u$ is strictly positive in a small neighborhood of  $\partial B_1$. Let $\tilde{\gamma}>0$ and $\eta(x)$ be the inward unit normal at $x,$ where $x \in \partial B_1.$ Now define 
$$N_{\tilde{\gamma}}(\partial B_1) = \{ x+ \tilde{\beta} \eta(x) : \tilde{\beta} \in [0,\tilde{\gamma}) , x \in \partial B_1 \}.$$ It can be seen that $N_{\tilde{\gamma}}(\partial B_1)$ is an open neighborhood of $\partial B_1.$ Also, $|N_{\tilde{\gamma}} (\partial B_1)| \rightarrow 0$ as $\tilde{\gamma} \rightarrow 0$. Now, $\tilde{\gamma} >0$ can be chosen sufficiently small so that $|N_{\tilde{\gamma}}(\partial B_1)| < \frac{C_1}{2}$. Define $ \widetilde{K}= (S \cap B(0,R)) \setminus N_{\tilde{\gamma}} (\partial B_1)$. Then it follows that $|\widetilde{K}|>\frac{C_1}{2} >0$.
	
Let $G(x,y)$ denote the Green's function for problem (\ref{Problem 1}). Let $\xi_0 \in \partial B_1$ and fix a element $x_0 \in \widetilde{K}$ arbitrarily. The boundary point lemma ensures the existence of a $b_{\xi_0}>0$ such that $\frac{\partial G}{\partial \eta }(x_0, \xi_0) \geq b_{\xi_0}$. As $\tilde{K} \times \partial B_1$ is compact, we have $b>0$ and $\tilde{\epsilon}>0$ for which $\frac{\partial G}{\partial \eta}(x, \xi) \geq b$ for all $x \in \widetilde{K}$ and $\xi \in N_{\tilde{\epsilon}}(\partial B_1)$. Hence, for all $x \in \widetilde{K}$ and $\xi$ with $d(\xi, \partial B_1) < \tilde{\epsilon}$, where $\xi \rightarrow d(\xi, \partial B_1)$ denotes the distance function, there exists $C>0$ satisfying the inequality $G(x, \xi) \geq C d(\xi , \partial B_1).$ 
	
Now, $u(x)$ can be written as $u(x) = \int_{B_1 ^c} \lambda K(y) f(u(y)) G(x,y) dy$ for $x \in B_1 ^c.$
Let $\xi$  be such that $d(\xi, \partial B_1) < \tilde{\epsilon}.$ Then 
\begin{align*}
u(\xi) &= \int_{\widetilde{K}} \lambda K(y) f(u(y)) G(\xi,y) dy+ \int_{B_1^c \setminus \widetilde{K}} \lambda K(y) f(u(y)) G(\xi,y) dy \\
& \geq \lambda C d(\xi , \partial B_1) \int_{\widetilde{K}} K(y) [g(u)+f(0)] dy + \lambda f(0) \int_{B_1^c \setminus \widetilde{K}} K(y) G(\xi, y) dy \\
& \geq \lambda C d(\xi , \partial B_1) \left[ \int_{\widetilde{K}} K(y) u^s dy + f(0)|\widetilde{K}| \| K \|_{L^{\infty}(B_1^c)} \right] + \lambda f(0)  \int_{B_1 ^c} K(y)G(\xi , y) dy .
\end{align*}
Note that $e(x) \coloneq \int_{B_1^c} K(y) G(x,y) dy$ is the solution of the following problem
\begin{equation}
\begin{cases}
- \Delta e(x) = K(x) \hspace{.1cm}\mbox{ in } B_1 ^c,  \\
e(x)= 0 \hspace{1.1cm} \mbox{ on } \partial B_1.
\end{cases}
\end{equation}
and $e(x)$ is bounded at infinity. Therefore, $ \int_{B_1^c} K(y)G(\xi,y) dy \leq C d(\xi, \partial B_1)$ for some constant $C>0$. Combining the above estimates, we get
\begin{align*}
u(\xi) & \geq \lambda C d(\xi , \partial B_1) \left[ \int_{\widetilde{K}} \epsilon ^s \lambda^{\frac{-s}{s-1}} K(y)+ f(0)|\widetilde{K}| \| K \|_{L^{\infty}(B_1^c)}   +  f(0)  \right] \\
& \geq \lambda^{\frac{-1}{s-1}} \left[C - C \lambda^{\frac{s}{s-1}} \right] d(\xi , \partial B_1) \\
&\geq C \lambda^{\frac{-1}{s-1}} d(\xi, \partial B_1), 
\end{align*}
for $\lambda > 0$ sufficiently small and for all $\xi$ with $d(\xi, \partial B_1) < \tilde{\epsilon}$. From the continuity of $u$ and the distance function, it follows that  for all $\xi$ with $d(\xi, \partial B_1) = \tilde{\epsilon},$ $u(\xi) \geq C \tilde{\epsilon} \lambda^{\frac{-1}{s-1}}.$

To prove the positivity of $u$ outside this neighborhood, we first prove the following Lemma. 
\begin{lemma} 
\label{lower estimate}
    Consider the problem 
    \begin{equation} \label{positivity_proof_eqn}
\begin{cases}
- \Delta u &= h(x)\hspace{.1cm} \mbox { in } B_1 ^c, \\
u&=0 \hspace{.1cm} \mbox { on }  \partial B_1, \\
\end{cases}
\end{equation}
where $h$ is a nonnegative function. Let $u$ be a non negative solution of \eqref{lower estimate}, then there exists $C>0$ such that $u(x)\geq C$ for all $|x|\geq r,$ with $r>1$. 
\end{lemma}

\begin{proof}
\noindent Define $s=\displaystyle{\min_{|x|=r}} \; u(x)(>0)$ and let $k\in \mathbb{N}$ with $k>r.$ Let 
\begin{equation}\label{u_k}
    u_k (x):= \frac{s\log k}{\log k - \log r } \left(1-\frac{\log |x|}{\log k}\right).
\end{equation}
It can be seen that $u_k$ is a radial solution for the following problem  
\begin{equation} \label{positivity_proof_eqn1}
\begin{cases}
- \Delta u_k(x) &= 0\hspace{.1cm} \mbox { in } r<|x|<k, \\
u_k(x)&=s \hspace{.1cm} \mbox { if }  |x|=r, \\
u_k(x)&=0 \hspace{.1cm} \mbox { if }  |x|=k. \\
\end{cases}
\end{equation}
Also, it may be noted that $v(x)=\displaystyle{\lim_{k\to \infty}} u_k(x)=s$ is a radial solution of the problem given below.
\begin{equation} \label{positivity_proof_eqn2}
\begin{cases}
- \Delta v &= 0\hspace{.1cm} \mbox { in } B_r^c, \\
v&=s \hspace{.1cm} \mbox { on }  \partial B_r. \\
\end{cases}
\end{equation}

\noindent For any $x \in B_r^c,$ there exists $k\in \mathbb{N}$ such that $r<|x|<k.$ By maximum principle, $u_k(x) \leq u(x).$ As this is true for all $k>|x|, \displaystyle{\lim_{k\to \infty}} u_k(x) \leq u(x).$ Hence $u(x) \geq s.$ Since the choice of $x\in B_r^c$ was arbitrary, we can conclude that $u(x) \geq s$ for all $|x| \geq r.$\\
\end{proof} 

Now, let $w>0$ be the solution of the following problem 
\begin{equation} \label{positivity_proof_eqn3}
\begin{cases}
- \Delta w = -f(0)K(x)\hspace{.1cm} \mbox { in } B_1^c, \\
w=0 \hspace{2.1cm} \mbox { on }  \partial B_1.
\end{cases}
\end{equation}
It follows that $- \Delta (u+\lambda w)= \lambda K(x)g(u) \geq 0.$ From Lemma \ref{lower estimate}, we get the existence of a positive constant $C$ such that $u(x)+\lambda w (x) \geq C $ for all $|x| \geq 1+\tilde{\epsilon}.$ Since $w$ is bounded by Theorem \ref{regularity}, there exists $\tilde{C}>0$ such that
\begin{equation*}
    \begin{aligned}
        u(x)&=u(x) + \lambda w(x)- \lambda w(x)\\
        &\geq C- \lambda \tilde{C},
    \end{aligned}
\end{equation*}
for all $|x| \geq 1+\tilde{\epsilon}.$ For sufficiently small $\lambda $, this implies $ u(x) > 0 $ for all $|x| \geq  1+\tilde{\epsilon}.$ 
\qed
\\

{\bf \noindent Acknowledgement } \\

\noindent This research was supported by the Core Research Grant, awarded by the Science and Engineering Research Board, Department of Science and Technology, Government of India (CRG/2022/008985).

\begin{thebibliography}{10}

\bibitem{MR1219715}
W.~Allegretto and P.~O. Odiobala, \emph{Nonpositone elliptic problems in {${\bf
  R}^n$}}, Proc. Amer. Math. Soc. \textbf{123} (1995), 533--541.

\bibitem{MR3347486}
T.~V. Anoop, Pavel Dr\'abek, and Sarath Sasi, \emph{Weighted quasilinear
  eigenvalue problems in exterior domains}, Calc. Var. Partial Differential
  Equations \textbf{53} (2015), 961--975.

\bibitem{MR1849197}
Marie-Francoise Bidaut-V\'eron and Stanislav Pohozaev, \emph{Nonexistence
  results and estimates for some nonlinear elliptic problems}, J. Anal. Math.
  \textbf{84} (2001), 1--49.

\bibitem{MR768824}
L.~Caffarelli, R.~Kohn, and L.~Nirenberg, \emph{First order interpolation
  inequalities with weights}, Compositio Math. \textbf{53} (1984), 259--275.

\bibitem{MR2328697}
Scott Caldwell, Alfonso Castro, Ratnasingham Shivaji, and Sumalee Unsurangsie,
  \emph{Positive solutions for classes of multiparameter elliptic semipositone
  problems}, Electron. J. Differential Equations (2007), No. 96, 10.

\bibitem{MR3925556}
Siegfried Carl, David~G. Costa, Morteza Fotouhi, and Hossein Tehrani,
  \emph{Invariance of critical points under {K}elvin transform and multiple
  solutions in exterior domains of {$\Bbb R^2$}}, Calc. Var. Partial
  Differential Equations \textbf{58} (2019), Paper No. 65, 24.

\bibitem{MR3801828}
Maya Chhetri, R.~Shivaji, Byungjae Son, and Lakshmi Sankar, \emph{An existence
  result for superlinear semipositone {$p$}-{L}aplacian systems on the exterior
  of a ball}, Differential Integral Equations \textbf{31} (2018), 643--656.

\bibitem{MR3415737}
R.~Dhanya, Q.~Morris, and R.~Shivaji, \emph{Existence of positive radial
  solutions for superlinear, semipositone problems on the exterior of a ball},
  J. Math. Anal. Appl. \textbf{434} (2016), 1533--1548.

\bibitem{MR1625845}
Lawrence~C. Evans, \emph{Partial differential equations}, Graduate Studies in
  Mathematics, vol.~19, American Mathematical Society, Providence, RI, 1998.

\bibitem{MR737190}
David Gilbarg and Neil~S. Trudinger, \emph{Elliptic partial differential
  equations of second order}, Springer-Verlag, Berlin, 1983.

\bibitem{MR2777537}
Qing Han and Fanghua Lin, \emph{Elliptic partial differential equations},
  Courant Institute of Mathematical Sciences, New York; American Mathematical
  Society, Providence, RI, 2011.

\bibitem{MR4500097}
Anumol Joseph and Lakshmi Sankar, \emph{Positive solutions to superlinear
  semipositone problems on the exterior of a ball}, Complex Var. Elliptic Equ.
  \textbf{67} (2022), 2633--2645.

\bibitem{MR3782023}
Quinn Morris, Ratnasingham Shivaji, and Inbo Sim, \emph{Existence of positive
  radial solutions for a superlinear semipositone {$p$}-{L}aplacian problem on
  the exterior of a ball}, Proc. Roy. Soc. Edinburgh Sect. A \textbf{148}
  (2018), 409--428.

\bibitem{MR109940}
L.~Nirenberg, \emph{On elliptic partial differential equations}, Ann. Scuola
  Norm. Sup. Pisa Cl. Sci. (3) \textbf{13} (1959), 115--162.

\bibitem{rabinowitz1986minimax}
Paul~H. Rabinowitz, \emph{Minimax methods in critical point theory with
  applications to differential equations}, vol.~65, Conference Board of the
  Mathematical Sciences, Washington, DC; by the American Mathematical Society,
  Providence, RI, 1986.

\bibitem{MR1454361}
C.~G. Simader and H.~Sohr, \emph{The {D}irichlet problem for the {L}aplacian in
  bounded and unbounded domains}, Pitman Research Notes in Mathematics Series,
  vol. 360, Longman, Harlow, 1996.

\bibitem{Tolksdorf-84}
Peter Tolksdorf, \emph{Regularity for a more general class of quasilinear
  elliptic equations}, J. Differential Equations \textbf{51} (1984), 126--150.

\bibitem{MR2636416}
Sumalee Unsurangsie, \emph{Existence of a solution for a wave equation and an
  elliptic {D}irichlet problem}, ProQuest LLC, Ann Arbor, MI, 1988, Thesis
  (Ph.D.)--University of North Texas.

\end{thebibliography}
\providecommand{\bysame}{\leavevmode\hbox to3em{\hrulefill}\thinspace}
\providecommand{\MR}{\relax\ifhmode\unskip\space\fi MR }
\providecommand{\MRhref}[2]{%
  \href{http://www.ams.org/mathscinet-getitem?mr=#1}{#2}
}
\providecommand{\href}[2]{#2}

\end{document}